\documentclass[a4paper,11pt,twoside,notitlepage]{amsart}

\usepackage{amsmath,amssymb,amsthm,amscd}
\usepackage{mathtools}
\usepackage{mathrsfs}

\usepackage{graphicx}
\usepackage{xcolor}
\usepackage{tikz}

\usepackage[ocgcolorlinks,linkcolor=blue]{hyperref}
\usepackage{url}
\usepackage[shortlabels]{enumitem}
\usepackage{comment}
\usepackage{verbatim}

\usepackage{bm}
\usepackage{bbm}
\usepackage{dsfont}
\usepackage{relsize}
\usepackage{esint}
\usepackage[normalem]{ulem}
\usepackage[utf8]{inputenc}

\newcommand\redsout{\bgroup\markoverwith{\textcolor{red}{\rule[0.5ex]{2pt}{0.4pt}}}\ULon}

\numberwithin{equation}{section}

\DeclareMathOperator{\Cof}{Cof}

\DeclareMathOperator{\Id}{Id}

\allowdisplaybreaks
\newcommand{\para}[1]{\vspace{3mm}\noindent\textbf{#1.}}

\mathtoolsset{showonlyrefs}
\graphicspath{{images/}}

\newtheorem{theorem}{Theorem}[section]
\newtheorem{lemma}[theorem]{Lemma}
\newtheorem*{lemma*}{Lemma}

\newtheorem{corollary}[theorem]{Corollary}

\newtheorem{proposition}[theorem]{Proposition}

\newtheorem{remark}[theorem]{Remark}
\newtheorem{assumption}{Assumption}[section]

\title[Gaussian curvature equation]{Gauge rigidity in an inverse problem for the prescribed Gaussian curvature equation}

\author[Y.-H. Lin]{Yi-Hsuan Lin}
\address{Department of Applied Mathematics, National Yang Ming Chiao Tung University, Hsinchu, Taiwan \& Fakult\"at f\"ur Mathematik, University of Duisburg-Essen, Essen, Germany}
\curraddr{}
\email{yihsuanlin3@gmail.com}

\subjclass[2020]{Primary 35R30; Secondary 35J60, 35J96, 53A05}

\keywords{inverse problems, prescribed Gaussian curvature equation, first linearization, anisotropic conductivity, gauge rigidity, unique continuation}

\newcommand{\C}{{\mathbb C}}
\newcommand{\R}{{\mathbb R}}

\newcommand {\p} {\partial}

\DeclareMathOperator{\loc}{loc} 

\begin{document}

	\begin{abstract}
		We prove uniqueness for an inverse boundary problem for the prescribed Gaussian curvature equation $\det D^2u=K(x)(1+|\nabla u|^2)^2$ on a smooth bounded simply connected planar domain. We compare smooth positive curvatures on a common open class of Dirichlet values for which the convex problems are well posed. Equality of the nonlinear Dirichlet-to-Neumann maps, together with the prescribed first boundary jet of the curvature, implies equality of the curvatures throughout the domain. The first linearized equation has conductivity $\Sigma=K(D^2u)^{-1}$. The planar anisotropic conductivity theorem determines this tensor up to a boundary fixing diffeomorphism. We prove that the Hessian and gradient structure forces this diffeomorphism to be the identity. The difference between the pulled-back background gradient and the original gradient satisfies a local elliptic system whose components have the same scalar principal part and zero boundary Cauchy data. A first-order Carleman estimate gives unique continuation for this system. The argument uses the nonlinear boundary response at one admissible Dirichlet value and its first derivative there.
	\end{abstract}
	
	\maketitle
	
	\tableofcontents

	\section{Introduction}\label{sec:introduction}
	
	The prescribed Gaussian curvature equation for graphs is a natural fully nonlinear equation in geometric analysis. If $u$ is locally strictly convex on a planar domain $\Omega\subset\mathbb R^2$, the graph $\{(x,u(x)):x\in\Omega\}\subset\mathbb R^3$ has Gaussian curvature $\frac{\det D^2u}{(1+|\nabla u|^2)^2}$. Prescribing this curvature gives
	\begin{equation}\label{eq:GC-main}
		\begin{cases}
			\det D^2u=K(x)(1+|\nabla u|^2)^2 & \text{in }\Omega,\\
			u=\varphi & \text{on }\partial\Omega.
		\end{cases}
	\end{equation}
	The inverse problem asks whether the boundary response of this nonlinear Dirichlet problem determines the interior function $K$.
	
	We first specify the forward class. The elliptic branch of \eqref{eq:GC-main} is tied to convexity, and one should not expect a solution map for arbitrary boundary values without an admissibility condition. We therefore compare the two curvatures only on boundary data for which both Dirichlet problems are well posed. We assume that there is a nonempty open set $\mathcal U\subset C^\infty(\partial\Omega)$ such that, for $j=1,2$ and $\varphi\in\mathcal U$,
	\begin{equation}\label{eq:GC-main-j-intro}
		\begin{cases}
			\det D^2u^{(j)}=K_j(x)(1+|\nabla u^{(j)}|^2)^2 & \text{in }\Omega,\\
			u^{(j)}=\varphi & \text{on }\partial\Omega
		\end{cases}
	\end{equation}
	has a unique smooth admissible solution $u_j(\varphi)$ with $D^2u_j(\varphi)>0$ on $\overline\Omega$. On this common class the nonlinear DN map is $\Lambda_{K_j}(\varphi)=\partial_\nu u_j(\varphi)|_{\partial\Omega}$.
	
	For a fixed admissible solution $u$, the logarithmically normalized first variation is $Lv=A^{ij}\partial_{ij}v+B^i\partial_i v$, where $A=(D^2u)^{-1}$ and $B=-4\nabla u/(1+|\nabla u|^2)$. The cofactor identity for a Hessian puts this equation in divergence form: $L=K^{-1}\operatorname{div}(\Sigma\nabla)$, with $\Sigma=KA=\operatorname{Cof}(D^2u)/(1+|\nabla u|^2)^2$. Boundary determination converts the first derivative of the nonlinear DN map into the conductivity DN map.
	
	The planar anisotropic conductivity theorem \cite{APL05} then gives a boundary fixing diffeomorphism $J$ with $\Sigma_2\circ J=(\det DJ)^{-1}DJ\,\Sigma_1DJ^T$. This is a gauge freedom for general conductivities. The conductivities arising from the prescribed curvature equation also have a Hessian and gradient structure, which imposes further conditions on $J$.
	
	We use these conditions in the original Euclidean coordinates. Let $p=(\nabla u_2)\circ J$ and $V=p-\nabla u_1$. Taking the cofactor of the conductivity transformation and applying the Piola identity gives $\operatorname{div}(a\nabla V_\alpha)+\partial_\alpha(c\cdot V)=0$ for $\alpha=1,2$, where $a$ is a smooth positive definite matrix and $c$ is a smooth vector field. Both equations have the same scalar principal part; the coupling has order at most one. Boundary fixing and equality of the background boundary gradients give zero Dirichlet and conormal data for $V$.
	
	We prove Cauchy uniqueness for this local system by reducing its common principal part to the Laplacian in local isothermal coordinates and applying a first-order Carleman estimate. This gives $V=0$. The cofactor transformation then implies $DJ=I$, and the boundary condition gives $J=\Id$. Thus the background solutions and their prescribed curvatures agree. Neither a smallness assumption on their difference nor global injectivity of their gradient maps is used.
	
	We now state the main theorem. Throughout the paper, $\Omega\subset\mathbb R^2$ is a bounded simply connected domain with $C^\infty$-smooth boundary.
	
	\begin{assumption}\label{assum:common-forward-solvability}
		Let $K_j\in C^\infty(\overline\Omega)$, $j=1,2$, satisfy
		$K_j\ge c_0>0$ on $\overline\Omega$ for some constant $c_0>0$. We assume
		that there exists a nonempty open set
		$\mathcal U\subset C^\infty(\partial\Omega)$ such that, for each $j=1,2$
		and each $\varphi\in\mathcal U$, the Dirichlet problem
		\begin{equation}\label{eq:common-forward-problem}
			\begin{cases}
				\det D^2u^{(j)}=K_j(x)\bigl(1+|\nabla u^{(j)}|^2\bigr)^2 & \text{in }\Omega,\\
				u^{(j)}=\varphi & \text{on }\partial\Omega
			\end{cases}
		\end{equation}
		has a unique solution
		$u_j(\varphi)\in C^\infty(\overline\Omega)$ with
		$D^2u_j(\varphi)>0$ in $\overline\Omega$.
	\end{assumption}
	
	Assumption~\ref{assum:common-forward-solvability} fixes the common domain on
	which the two nonlinear DN maps are defined. It is not used as an additional
	interior rigidity input. Under this assumption, for each
	$\varphi\in\mathcal U$ we define
	\begin{equation}\label{def:DN-map}
		\Lambda_{K_j}(\varphi)=\partial_\nu u_j(\varphi)|_{\partial\Omega},\quad j=1,2.
	\end{equation}
	
	\begin{theorem}\label{thm:main}
		Let $\Omega\subset\mathbb R^2$ be a bounded simply connected domain with
		$C^\infty$-smooth boundary. Let $K_j\in C^\infty(\overline\Omega)$,
		$j=1,2$, satisfy Assumption~\ref{assum:common-forward-solvability}, and let
		$\Lambda_{K_j}$ be the nonlinear DN maps defined on the common open set
		$\mathcal U$. Suppose that the boundary derivatives of $K_1$ and $K_2$
		agree up to order $1$, in the sense that, in every boundary normal coordinate
		system,
		\begin{equation}\label{eq:BC for the curvature}
			\partial_\tau^a\partial_\nu^b K_1
			=
			\partial_\tau^a\partial_\nu^b K_2
			\quad
			\text{on }\partial\Omega
			\quad
			\text{for all }a,b\geq0,\ a+b\leq1.
		\end{equation}
		Suppose also that
		\begin{equation}\label{DN map same}
			\Lambda_{K_1}(\varphi)=\Lambda_{K_2}(\varphi)
			\quad
			\text{for all }\varphi\in\mathcal U.
		\end{equation}
		Then $K_1=K_2$ in $\Omega$.
	\end{theorem}
	
	The boundary values of $K_j$ enter the recovery of the boundary conductivity. The proof below then eliminates the conductivity diffeomorphism using the local system for $V$. The first normal derivative of $K_j$ in \eqref{eq:BC for the curvature} is not needed for this argument; the precise boundary information used is recorded after the proof of the theorem.

	\para{Earlier literature}
	Linearization is a standard way to bring nonlinear inverse problems back to linear boundary value problems; an early example is \cite{isakov1993uniqueness_parabolic}. In many equations, the first variation is not enough, and one has to use higher-order terms in the boundary response. Second-order arguments occur, for instance, in \cite{AYT2017direct, CNV2019reconstruction, KN002, sun1996quasilinear, sun2010inverse, sun1997inverse}. Nonlinear interaction methods were developed in a systematic form for wave equations in \cite{KLU2018}, and they now appear in a wide range of nonlinear inverse problems.
	
	For elliptic equations, semilinear problems have been studied in \cite{FO20, LLLS2019nonlinear, LLLS2019partial, LLST2022inverse, KU2019remark, KU2019partial, FLL2021inverse}, including partial data results. Quasilinear elliptic problems are treated in \cite{KKU2022partial, CFKKU2021calderon, LW24_quasi, liimatainen2026inverse}. The minimal surface equation has been considered in \cite{ABN_20_minimal, carstea2024inverse, CLLT2023inverse_minimal, CLT24, nurminen2023inverse}. Further related results for nonlinear elliptic and fractional elliptic equations can be found in \cite{LL2019global, LL2020inverse, LSX_22_IP, harrach2022simultaneous, ST_23_single, LL25_book}; see also the survey \cite{lassas2025introduction}.
	
	The fully nonlinear case has some different features. The first variation is often a non-divergence form equation, and higher variations naturally involve Hessians of linearized solutions. In \cite{LL2025IP_Monge_Ampere} an inverse source problem for the Monge--Amp\`ere equation was studied in the plane, and another fully nonlinear inverse source problem was considered in \cite{LLW26_fully_nonlinear}. A related Monge--Amp\`ere inverse source problem is treated in \cite{CG_2026_MA}. For the prescribed Gaussian curvature equation, we use the special conductivity form of the first variation. The additional compatibility between the background Hessian and gradient removes the conductivity gauge through a local elliptic system. Thus the background response and the first variation suffice, without a second interaction identity.

	\para{Organization of the article} 
	Section~\ref{sec:prel} establishes smooth dependence on the Dirichlet data and the boundary reconstruction. Section~\ref{sec:first-linearization} derives the first linearized equation, identifies its conductivity DN map, and records the conductivity and drift gauge identities. Section~\ref{sec:local-Cauchy-uniqueness} proves interior and boundary unique continuation for systems with a common scalar principal part. Section~\ref{sec:ucp-and-gauge-elimination} derives the local equation for the pulled-back gradient, determines the gauge, and proves Theorem~\ref{thm:main}.

	\section{Preliminaries}\label{sec:prel}
	
	\subsection{The solvability framework}
	
	We begin with the forward class used for the linearization argument. Throughout the paper, $\Omega\subset\mathbb R^2$ is bounded and simply connected with $C^\infty$ boundary, and the pair $K_1,K_2$ satisfies Assumption~\ref{assum:common-forward-solvability}.
	
	For a fixed curvature $K$, the equation can be written as $\det D^2u=\psi(x,u,\nabla u)$ with $\psi(x,z,p)=K(x)(1+|p|^2)^2$. The function $\psi$ is smooth, positive, and independent of $z$. On the common solvability class, the solutions are smooth and admissible, and the Hessian remains positive up to the boundary. All linearizations below are taken along such nondegenerate branches.
	
	\begin{proposition}[Local smooth dependence]\label{prop:wellposedness}
		Assume that the pair satisfies Assumption~\ref{assum:common-forward-solvability}. Fix $j\in\{1,2\}$, $\phi\in\mathcal U$, an integer $m\geq2$, and $\alpha\in(0,1)$.
		\begin{enumerate}[\rm(i)]
			\item\label{item:wellposed-existence} There is a unique smooth admissible solution $u_j(\phi)$ of
			\begin{equation}\label{eq:wellposed-background}
				\begin{cases}
					\det D^2u_j(\phi)=K_j(x)(1+|\nabla u_j(\phi)|^2)^2&\text{in }\Omega,\\
					u_j(\phi)=\phi&\text{on }\partial\Omega.
				\end{cases}
			\end{equation}
			Its Hessian is uniformly positive definite on $\overline\Omega$.
			
			\item\label{item:wellposed-smooth-dependence} There is an open neighborhood $\mathcal U_{m,\alpha}$ of $\phi$ in $C^{m,\alpha}(\partial\Omega)$ and a smooth Banach space map $S_{j,m,\alpha}:\mathcal U_{m,\alpha}\to C^{m,\alpha}(\overline\Omega)$ such that $S_{j,m,\alpha}(\phi)=u_j(\phi)$ and, for every $\varphi\in\mathcal U_{m,\alpha}$, the function $u=S_{j,m,\alpha}(\varphi)$ is admissible and satisfies
			\begin{equation}\label{eq:wellposed-nearby}
				\begin{cases}
					\det D^2u=K_j(x)(1+|\nabla u|^2)^2&\text{in }\Omega,\\
					u=\varphi&\text{on }\partial\Omega.
				\end{cases}
			\end{equation}
			For $\varphi\in\mathcal U\cap\mathcal U_{m,\alpha}$, this is the solution in Assumption~\ref{assum:common-forward-solvability}. The local maps are compatible on different H\"older scales. On $\mathcal U$, the maps $\varphi\mapsto u_j(\varphi)$ and $\varphi\mapsto\Lambda_{K_j}(\varphi)$ are smooth in the $C^\infty$ topology: their iterated directional derivatives exist and are jointly continuous in the base point and directions.
		\end{enumerate}
	\end{proposition}
	
	\begin{proof}
		The first assertion is part of Assumption~\ref{assum:common-forward-solvability}. The smallest eigenvalue of the continuous positive definite matrix $D^2u_j(\phi)$ has a positive minimum on the compact set $\overline\Omega$. Thus the positivity is uniform.
		
		For the local map, let $X_m=\{w\in C^{m,\alpha}(\overline\Omega):w|_{\partial\Omega}=0\}$ and choose a bounded linear right inverse $E_m:C^{m,\alpha}(\partial\Omega)\to C^{m,\alpha}(\overline\Omega)$ of the trace. Write $u=u_j(\phi)+E_m(\varphi-\phi)+w$ and define
		\begin{equation}\label{eq:local-nonlinear-map}
			\mathcal F_m(w,\varphi)=\det D^2u-K_j(x)(1+|\nabla u|^2)^2.
		\end{equation}
		This map takes $X_m\times C^{m,\alpha}(\partial\Omega)$ into $C^{m-2,\alpha}(\overline\Omega)$. The map $(w,\varphi)\mapsto u$ is continuous and affine, and spatial differentiation through order two is continuous between the stated H\"older spaces. The determinant is a polynomial in the second derivatives, and $(1+|\nabla u|^2)^2$ is a polynomial in the first derivatives. Multiplication is continuous in $C^{m-2,\alpha}$, including when $m=2$. Consequently, $\mathcal F_m$ is a smooth map of real Banach spaces and $\mathcal F_m(0,\phi)=0$.
		
		Its derivative in $w$ at $(0,\phi)$ is
		\begin{equation}\label{eq:local-Dirichlet-linearization}
			Pv=\Cof(D^2u_j(\phi))^{ab}\partial_{ab}v-4K_j(1+|\nabla u_j(\phi)|^2)\nabla u_j(\phi)\cdot\nabla v.
		\end{equation}
		The cofactor matrix is uniformly positive definite, and all coefficients are smooth. The operator has no zeroth-order term. The maximum principle gives uniqueness for $Pv=0$ with zero boundary value. Dirichlet solvability for a uniformly elliptic operator on a smooth domain and the Schauder estimate give, for every $f\in C^{m-2,\alpha}$, a unique $v\in X_m$ with $Pv=f$ and
		\begin{equation}\label{eq:local-Dirichlet-bound}
			\|v\|_{C^{m,\alpha}(\overline\Omega)}\leq C\|f\|_{C^{m-2,\alpha}(\overline\Omega)}.
		\end{equation}
		The $C^0$ term in the Schauder estimate is controlled by the maximum principle estimate. For $m>2$, apply higher-order boundary regularity after solving at order two. Thus, $P:X_m\to C^{m-2,\alpha}$ is an isomorphism.
		
		The Banach implicit function theorem now gives $w=w(\varphi)$ smoothly on an open neighborhood $\mathcal U_{m,\alpha}$ of $\phi$. Define $S_{j,m,\alpha}(\varphi)=u_j(\phi)+E_m(\varphi-\phi)+w(\varphi)$. Continuity into $C^2$ permits shrinking this neighborhood so that its Hessians remain uniformly positive definite.
		
		We justify identification with the assumed smooth branch without presuming that it already depends continuously on $\varphi$. If $u,v\in C^{2,\alpha}(\overline\Omega)$ are two admissible solutions with the same boundary value, set $z=u-v$ and $u_t=v+tz$. Their logarithmic equations and the fundamental theorem of calculus give
		\begin{equation}\label{eq:admissible-comparison}
			\bigg(\int_0^1(D^2u_t)^{-1}\,dt\bigg)^{ab}\partial_{ab}z-4\bigg(\int_0^1\frac{\nabla u_t}{1+|\nabla u_t|^2}\,dt\bigg)\cdot\nabla z=0.
		\end{equation}
		Every $D^2u_t$ is positive definite, with uniform bounds on $\overline\Omega\times[0,1]$. This is a uniformly elliptic equation for $z$ with bounded coefficients, no zeroth-order term, and zero boundary value. The maximum principle gives $z=0$. It follows that the local Banach solution coincides with $u_j(\varphi)$ whenever $\varphi\in\mathcal U\cap\mathcal U_{m,\alpha}$. It also follows that the maps obtained at different H\"older levels coincide wherever both are defined and admissible.
		
		For smooth data in $\mathcal U$, the values of these maps are smooth by the original solvability assumption. At each smooth base point, the preceding construction can be performed at every $m$. Banach derivatives of the maps at different levels agree after inclusion into $C^{2,\alpha}$, since the maps agree on their common domains and differentiation commutes with continuous linear inclusions. Therefore, on smooth directions, an iterated derivative has values in every $C^{m,\alpha}$. Its joint continuity in each of these norms follows from the corresponding Banach derivative. This proves smoothness in the Fr\'echet topology; no common neighborhood in all higher norms is assumed.
		
		The maps $T_0u=u|_{\partial\Omega}$ and $T_\nu u=\nu\cdot\nabla u|_{\partial\Omega}$ are continuous linear maps from $C^{m,\alpha}(\overline\Omega)$ to $C^{m,\alpha}(\partial\Omega)$ and $C^{m-1,\alpha}(\partial\Omega)$, respectively. They commute with parameter differentiation. In particular,
		\begin{equation}\label{eq:solution-map-trace-derivatives}
			D^k\Lambda_{K_j}(\varphi)[h_1,\ldots,h_k]=T_\nu D^kS_j(\varphi)[h_1,\ldots,h_k],\quad k\geq1.
		\end{equation}
		The Hessian map is likewise continuous linear from $C^{m,\alpha}$ to $C^{m-2,\alpha}$. Hence, differentiation of the equation and its boundary traces along any smooth finite-dimensional family is justified at the required level. This completes the proof.
	\end{proof}
	
	\subsection{Boundary determination}
	
	The boundary reconstruction determines the background gradient and Hessian from the Dirichlet value, its measured normal derivative, and $K|_{\partial\Omega}$. It determines the boundary conductivity of the first linearized equation.
	
	\begin{lemma}[Boundary determination]\label{lem:boundary-determination}
		Assume the hypotheses of Theorem~\ref{thm:main}. Then, for each $\varphi\in\mathcal U$, the nonlinear DN map together with $K_j|_{\partial\Omega}$ determines the boundary values of
		\begin{equation*}
			A_\varphi^{(j)}=(D^2u_\varphi^{(j)})^{-1},	\quad B_\varphi^{(j)i}
			=-\frac{4\partial_i u_\varphi^{(j)}}{1+|\nabla u_\varphi^{(j)}|^2},	\quad j=1,2.
		\end{equation*}
		In particular, if \eqref{eq:BC for the curvature} and \eqref{DN map same} hold, then $A_\varphi^{(1)}=A_\varphi^{(2)}$ on $\partial\Omega$.
	\end{lemma}
	
	\begin{proof}
		Fix $j\in\{1,2\}$, $\varphi\in\mathcal U$, and $p\in\partial\Omega$. Choose boundary normal coordinates $(s,t)$ near $p$, where $s$ is arclength on $\partial\Omega$ and $t$ is the inward normal coordinate. We write $\tau$ for the unit tangent vector and $\nu_{\rm in}$ for the inward unit normal.
		
		The Dirichlet datum gives $u_\varphi^{(j)}=\varphi$ on the boundary, hence, all tangential derivatives of $u_\varphi^{(j)}$ there are known. The nonlinear DN map gives the outer normal derivative, and also $\partial_tu_\varphi^{(j)}|_{\partial\Omega}$ after fixing the sign convention. Taking tangential derivatives, we know $u_\varphi^{(j)}$, $\partial_su_\varphi^{(j)}$, $\partial_tu_\varphi^{(j)}$, $\partial_{ss}u_\varphi^{(j)}$, and $\partial_{st}u_\varphi^{(j)}$ on $\partial\Omega$. 
		
		The only second-order boundary component still unknown is the normal-normal derivative. At $p$, write the Euclidean Hessian in the orthonormal frame $(\tau,\nu_{\rm in})$ as
		\begin{equation*}
			H_{\tau\tau}=D^2u_\varphi^{(j)}(\tau,\tau),\quad 
			H_{\tau\nu}=D^2u_\varphi^{(j)}(\tau,\nu_{\rm in}),\quad
			H_{\nu\nu}=D^2u_\varphi^{(j)}(\nu_{\rm in},\nu_{\rm in}).
		\end{equation*}
		Let $\kappa$ be the signed curvature defined by $\partial_s\tau=\kappa\nu_{\rm in}$, so $\partial_s\nu_{\rm in}=-\kappa\tau$. Put $g_{\rm in}=\nabla u_\varphi^{(j)}\cdot\nu_{\rm in}=-\Lambda_{K_j}(\varphi)$. Differentiating along the boundary gives $H_{\tau\tau}=\partial_s^2\varphi-\kappa g_{\rm in}$ and $H_{\tau\nu}=\partial_sg_{\rm in}+\kappa\partial_s\varphi$. Every quantity on the right is known. The normal coordinate lines are straight, so the remaining component is $H_{\nu\nu}=\partial_{tt}u_\varphi^{(j)}$.
		
		In the frame $(\tau,\nu_{\rm in})$, $\det D^2u_\varphi^{(j)}=H_{\tau\tau}H_{\nu\nu}-H_{\tau\nu}^2$. Evaluating the equation on $\partial\Omega$ gives
		\begin{equation}\label{eq:boundary-normal-normal-recovery}
			H_{\tau\tau}\,\partial_{tt}u_\varphi^{(j)}-H_{\tau\nu}^2=K_j(1+|\nabla u_\varphi^{(j)}|^2)^2
			\quad \text{on }\partial\Omega.
		\end{equation}
		The right-hand side is known from $K_j|_{\partial\Omega}$ and the known first boundary jet of $u_\varphi^{(j)}$. Since the solution is admissible, $D^2u_\varphi^{(j)}>0$ on $\overline\Omega$, $H_{\tau\tau}>0$ on $\partial\Omega$. Thus, \eqref{eq:boundary-normal-normal-recovery} uniquely determines $\partial_{tt}u_\varphi^{(j)}$ on the boundary.
		
		We have recovered the full second-order boundary jet of
		$u_\varphi^{(j)}$. Hence
		$A_\varphi^{(j)}=(D^2u_\varphi^{(j)})^{-1}$ is determined on
		$\partial\Omega$. The coefficient
		$B_\varphi^{(j)}$ depends only on the first boundary jet, so it is also
		determined.
		
		If \eqref{eq:BC for the curvature} and \eqref{DN map same} hold, then
		$K_1=K_2$ on $\partial\Omega$ and the two nonlinear DN maps give the same first boundary jet for the two solutions with the same boundary value $\varphi$. The reconstruction above, therefore, gives the same second-order boundary Hessian for $u_\varphi^{(1)}$ and $u_\varphi^{(2)}$. Hence,
		$A_\varphi^{(1)}=A_\varphi^{(2)}$ on $\partial\Omega$.
	\end{proof}
	
	\begin{remark}\label{rem:normal-derivative-curvature-not-used-boundary-determination}
		Lemma~\ref{lem:boundary-determination} uses only $K_j|_{\partial\Omega}$, not normal derivatives of $K_j$. The local rigidity argument below also requires no normal derivative of the curvature. The stronger boundary assumption in Theorem~\ref{thm:main} is therefore sufficient but not necessary for the proof.
	\end{remark}

	\section{The first linearization}\label{sec:first-linearization}
	
	We first linearize the equation in logarithmic form. This normalization removes a harmless positive factor from the principal coefficient and leaves the gradient dependence as a drift term.
	
	To compare this with Proposition~\ref{prop:wellposedness}, differentiate the determinant form. The principal coefficient is $\operatorname{Cof}(D^2u_\varphi)$. Since $\operatorname{Cof}(D^2u_\varphi)=(\det D^2u_\varphi)(D^2u_\varphi)^{-1}$ and $\det D^2u_\varphi>0$ on the admissible branch, multiplication by $(\det D^2u_\varphi)^{-1}$ gives the operator used below. This multiplication does not change the Dirichlet problem, the first variations, or the linearized boundary data.
	
	For the moment, let the curvature be a single function $K$. If $u_\varphi$ is the admissible solution with boundary value $\varphi$, then the equation is equivalently
	\[
	\log\det D^2u_\varphi-2\log(1+|\nabla u_\varphi|^2)-\log K=0.
	\] 
	For matrices of the same size we write $A:B=\operatorname{tr}(A^TB)$; for symmetric matrices this is $A:B=A^{ij}B_{ij}$. Repeated indices are summed over $1,2$.
	
	\begin{lemma}[Formula for the first linearized operator]
		\label{lem:first-lin-formula}
		Let $\varphi\in\mathcal U$ and $h\in C^\infty(\partial\Omega)$. For $|t|$
		sufficiently small, $\varphi+th\in\mathcal U$. Define $v_h=\frac{d}{dt}\big|_{t=0}u_{\varphi+th}$, then $v_h$ satisfies
		\begin{equation*}
			\begin{cases}
				L_\varphi v_h=0 & \text{in }\Omega,\\
				v_h=h & \text{on }\partial\Omega,
			\end{cases}
		\end{equation*}
		where
		\begin{equation}\label{eq:first-lin-op}
			L_\varphi=A_\varphi^{ij}\partial_{ij}+B_\varphi^i\partial_i,
		\end{equation}
		with
		\begin{equation}\label{eq:A_varphi and B_varphi}
			A_\varphi=(D^2u_\varphi)^{-1},
			\quad
			B_\varphi^i=-\frac{4\partial_i u_\varphi}{1+|\nabla u_\varphi|^2}.
		\end{equation}
	\end{lemma}
	
	\begin{proof}
		Let $u_t=u_{\varphi+th}$. Proposition~\ref{prop:wellposedness} gives differentiability of $t\mapsto u_t$ in every H\"older space under consideration. Since $T_0$ is continuous linear, differentiation of $T_0u_t=\varphi+th$ gives $T_0v_h=h$. The same proposition gives $\partial_\nu v_h|_{\partial\Omega}=D\Lambda_K(\varphi)[h]$.
		
		The derivative of $\log\det H$ at a positive definite matrix is $H^{-1}:\dot H$. Also, $\frac{d}{dt}\log(1+|\nabla u_t|^2)=2\frac{\nabla u_t\cdot\nabla\dot u_t}{1+|\nabla u_t|^2}$. Differentiating the logarithmic equation in $C^{m-2,\alpha}$ gives
		\[
		(D^2u_\varphi)^{-1}:D^2v_h-\frac{4\nabla u_\varphi}{1+|\nabla u_\varphi|^2}\cdot\nabla v_h=0.
		\]
		This proves \eqref{eq:first-lin-op} and \eqref{eq:A_varphi and B_varphi}. The derivative identities hold in each H\"older space and thus for the smooth solutions used below.
	\end{proof}
	
	We shall also use the following geometric form of the first linearized operator. We use the convention
	\begin{equation*}
		\Delta_gv=g^{ij}\partial_{ij}v-g^{ij}\Gamma_{ij}^k(g)\partial_kv.
	\end{equation*}
	
	\begin{lemma}\label{lem:geom-rewrite}
		Let $g_\varphi$ be the Riemannian metric whose inverse coefficients are $(g_\varphi)^{ij}=A_\varphi^{ij}$, then 
		\begin{equation}\label{eq:first-lin-geom}
			L_\varphi=\Delta_{g_\varphi}+Y_\varphi\cdot\nabla,
		\end{equation}
		where
		\begin{equation*}
			Y_\varphi=B_\varphi+X_{g_\varphi},
			\quad X_{g_\varphi}^k=(g_\varphi)^{ij}\Gamma_{ij}^k(g_\varphi).
		\end{equation*}
	\end{lemma}
	
	\begin{proof}
		By the convention above,
		\begin{equation*}
			\Delta_{g_\varphi}v=A_\varphi^{ij}\partial_{ij}v
			-A_\varphi^{ij}\Gamma_{ij}^k(g_\varphi)\partial_kv
			=A_\varphi^{ij}\partial_{ij}v
			-X_{g_\varphi}\cdot\nabla v.
		\end{equation*}
		Substituting this into \eqref{eq:first-lin-op} gives $L_\varphi v=\Delta_{g_\varphi}v+(B_\varphi+X_{g_\varphi})\cdot\nabla v$.
	\end{proof}
	
	We compare the first linearized operators corresponding to $K_1$ and $K_2$. For $j=1,2$, let $u_\varphi^{(j)}$ be the solution with boundary value $\varphi$, and let $A_\varphi^{(j)}$, $B_\varphi^{(j)}$, $g_\varphi^{(j)}$, and $Y_\varphi^{(j)}$ denote the corresponding first linearized coefficients and geometric data.
	
	For the linearized inverse problem, the boundary operator is the conormal trace
	associated with the principal tensor. Thus, if $v_h^{(j)}$ solves
	\begin{equation}\label{eq:first-lin-equ-j=12}
		\begin{cases}
			L_\varphi^{(j)}v_h^{(j)}=0 & \text{in }\Omega,\\
			v_h^{(j)}=h & \text{on }\partial\Omega,
		\end{cases}
	\end{equation}
	we define 
	\begin{equation}\label{eq:first_DN_map}
		\Lambda_j': C^\infty(\p\Omega)\to C^\infty(\p\Omega), \quad h\mapsto A_\varphi^{(j)}\nabla v_h^{(j)}\cdot\nu\big|_{\partial\Omega},
	\end{equation}
	for $j=1,2$. The next lemma identifies this conormal trace with the linearized boundary data coming from the nonlinear Euclidean DN map.
	
	\begin{lemma}[Identification of the first linearized DN map]
		\label{lem:first-lin-Cauchy-data}
		Adopt the assumptions of Theorem~\ref{thm:main}. Fix
		$\varphi\in\mathcal U$. Then \eqref{eq:BC for the curvature} and
		\eqref{DN map same} imply
		\begin{equation}\label{eq:same_linearized_DN_map}
			\Lambda_1'h=\Lambda_2'h \quad 	\text{on }\partial\Omega
		\end{equation}
		for all $h\in C^\infty(\partial\Omega)$.
	\end{lemma}
	
	\begin{proof}
		By Lemma~\ref{lem:boundary-determination}, the nonlinear DN map together with $K_1=K_2$ on $\partial\Omega$ determines the boundary values of the logarithmically normalized principal coefficients. Hence,
		\begin{equation}\label{eq:boundary-A-equality-first-lin-proof}
			A_\varphi^{(1)}	=A_\varphi^{(2)}\quad\text{on }\partial\Omega.
		\end{equation}
		Only the boundary values of the curvature are used for this conormal identification.
		
		Let $u_\varepsilon^{(j)}$ be the solution corresponding to $K_j$ with boundary value $\varphi+\varepsilon h$. Since $\mathcal U$ is open, this path lies in $\mathcal U$ for all sufficiently small $|\varepsilon|$. The equality of the nonlinear DN maps gives $\partial_\nu u_\varepsilon^{(1)}=\partial_\nu u_\varepsilon^{(2)}$ on $\partial\Omega$. Differentiating at $\varepsilon=0$ yields
		\begin{equation}\label{eq:normal-derivative-first-variation-equality}
			\partial_\nu v_h^{(1)}=\partial_\nu v_h^{(2)}\quad	\text{on }\partial\Omega.
		\end{equation}
		
		Since both first variations have the same Dirichlet value $h$, their
		tangential derivatives agree on the boundary. Together with \eqref{eq:normal-derivative-first-variation-equality}, this gives $\nabla v_h^{(1)}=\nabla v_h^{(2)}$ on $\partial\Omega$. Combining this gradient equality with
		\eqref{eq:boundary-A-equality-first-lin-proof} gives $A_\varphi^{(1)}\nabla v_h^{(1)}\cdot\nu=A_\varphi^{(2)}\nabla v_h^{(2)}\cdot\nu$ on $\partial\Omega$.	This is exactly $\Lambda_1'h=\Lambda_2'h$.
	\end{proof}
	
	We use the planar anisotropic conductivity theorem directly. For a smooth positive definite symmetric matrix field $\Sigma$, let $\Lambda_\Sigma h=(\Sigma\nabla v)\cdot\nu|_{\partial\Omega}$, where $\operatorname{div}(\Sigma\nabla v)=0$ and $v|_{\partial\Omega}=h$. This is the Euclidean flux, not a metric unit normal derivative. The conductivity form \eqref{eq:first-linearized-divergence-form}, derived in the next subsection directly from the curvature equation, will identify this map with the measured first variation.
	
	\begin{lemma}[Planar conductivity gauge]\label{lem:imported-planar-gauge}
		Let $\Omega\subset\R^2$ be bounded and simply connected with smooth boundary. Let $\Sigma_j$ be smooth uniformly positive definite symmetric matrix fields, and let $K_j$ be smooth positive functions on $\overline\Omega$. Set $L_j=K_j^{-1}\operatorname{div}(\Sigma_j\nabla)=A_j^{ab}\partial_{ab}+B_j^b\partial_b$, where $A_j=K_j^{-1}\Sigma_j$ and $B_j^b=K_j^{-1}\partial_a\Sigma_j^{ab}$. Suppose that $\Lambda_{\Sigma_1}=\Lambda_{\Sigma_2}$, $\Sigma_1=\Sigma_2$ on $\partial\Omega$, and $K_1=K_2$ on $\partial\Omega$. Then there is an orientation preserving smooth diffeomorphism $J:\overline\Omega\to\overline\Omega$ fixing the boundary such that, with $N=DJ$ and $d=\det N$,
		\begin{equation}\label{eq:conductivity-pushforward}
			\Sigma_2\circ J=d^{-1}N\Sigma_1N^T.
		\end{equation}
		The positive function $\rho=d(K_2\circ J)/K_1$ equals one on the boundary and satisfies
		\begin{equation}\label{eq:imported-gauge-intertwining}
			L_1(f\circ J)=\rho(L_2f)\circ J\quad\text{for every }f\in C^\infty(\overline\Omega).
		\end{equation}
		In particular,
		\begin{equation}\label{eq:imported-gauge-principal}
			NA_1N^T=\rho(A_2\circ J),
		\end{equation}
		and, for $\alpha=1,2$,
		\begin{equation}\label{eq:imported-gauge-drift}
			A_1^{ij}\partial_{ij}J^\alpha+B_1^i\partial_iJ^\alpha=\rho(B_2^\alpha\circ J).
		\end{equation}
	\end{lemma}
	
	\begin{proof}
		By the planar anisotropic conductivity theorem \cite[Theorem 1]{APL05}, equality of the conductivity DN maps gives a boundary fixing quasiconformal homeomorphism satisfying \eqref{eq:conductivity-pushforward} almost everywhere. We first verify the regularity needed here. Taking determinants in that identity gives $\det\Sigma_2\circ J=\det\Sigma_1$. Define the smooth metrics $\mathfrak g_j=\sqrt{\det\Sigma_j}\,\Sigma_j^{-1}$. They have determinant one. Inversion of the matrix identity gives
		\begin{equation}\label{eq:normalized-metric-conformality}
			(DJ)^T(\mathfrak g_2\circ J)DJ=(\det DJ)\mathfrak g_1
		\end{equation}
		almost everywhere. Thus $J$ is conformal between these two smooth metrics.
		
		The smooth global isothermal coordinate theorem for a simply connected compact surface with smooth boundary gives conformal diffeomorphisms $T_j:(\overline\Omega,\mathfrak g_j)\to\overline{\mathbb D}$; see \cite[Theorems 3.4.15 and 3.4.16]{PSU23}. Equivalently, one may compose global isothermal coordinates with a Riemann map using smooth boundary regularity. The map $T_2\circ J\circ T_1^{-1}$ is a quasiconformal homeomorphism of the disk whose complex dilatation vanishes almost everywhere by \eqref{eq:normalized-metric-conformality}. The Cauchy--Riemann equation in distributions implies that it is holomorphic. Being bijective, it is a disk automorphism and hence extends smoothly with nonzero derivative to the closed disk. Composing back shows that $J$ and $J^{-1}$ are smooth on $\overline\Omega$. In particular, $d>0$ on $\overline\Omega$, and \eqref{eq:conductivity-pushforward} holds everywhere by continuity.
		
		At a boundary point, use the same ordered orthonormal tangent-normal basis in the source and target. Boundary fixing gives $N\tau=\tau$, so $N=\left(\begin{smallmatrix}1&r\\0&s\end{smallmatrix}\right)$, where $s=d>0$. Write the common boundary conductivity as $\left(\begin{smallmatrix}a&b\\b&c\end{smallmatrix}\right)$, with $c>0$. Equation \eqref{eq:conductivity-pushforward} is $N\Sigma_1N^T=s\Sigma_1$ at this point. Its lower right entry gives $s^2c=sc$, hence $d=s=1$ on the boundary. Since $K_1=K_2$ there, $\rho=1$ there as well.
		
		We next derive the operator identity, including the density. Under $y=J(x)$, the divergence transformation and \eqref{eq:conductivity-pushforward} give
		\begin{equation}\label{eq:conductivity-divergence-transform}
			\operatorname{div}_x\big(\Sigma_1\nabla_x(f\circ J)\big)
			=d\big[\operatorname{div}_y(\Sigma_2\nabla_yf)\big]\circ J.
		\end{equation}
		For example, this follows by writing $\Sigma_1N^T=dN^{-1}(\Sigma_2\circ J)$, differentiating its product with $(\nabla_yf)\circ J$, and using $\partial_i(dN^{-1})_{i\alpha}=0$, the Piola identity. Division by $K_1$ gives \eqref{eq:imported-gauge-intertwining} with the stated $\rho$.
		
		Multiplying \eqref{eq:conductivity-pushforward} by $d/K_1$ gives \eqref{eq:imported-gauge-principal}. Finally, the chain rule is
		\[
		\partial_{ij}(f\circ J)=((\partial_{\alpha\beta}f)\circ J)\partial_iJ^\alpha\partial_jJ^\beta+((\partial_\alpha f)\circ J)\partial_{ij}J^\alpha.
		\]
		Substituting this and the first derivative formula in \eqref{eq:imported-gauge-intertwining}, and comparing the coefficients of $\partial_\alpha f$, proves \eqref{eq:imported-gauge-drift}.
	\end{proof}
	
	A boundary fixing gauge pair is a smooth diffeomorphism $J_\varphi:\overline\Omega\to\overline\Omega$ with $J_\varphi|_{\partial\Omega}=\Id$ and a smooth positive function $\rho_\varphi$ with $\rho_\varphi|_{\partial\Omega}=1$. We write $I$ for the $2\times2$ identity matrix.
	
	\begin{proposition}[Unique determination up to gauge]\label{thm:fixed-background-gauge}
		Under the assumptions of Theorem~\ref{thm:main}, for any $\varphi\in\mathcal U$ there is a boundary fixing gauge pair $(J_\varphi,\rho_\varphi)$ such that
		\begin{equation}\label{eq:fixed-background-principal-gauge}
			DJ_\varphi A_\varphi^{(1)}DJ_\varphi^T=\rho_\varphi(A_\varphi^{(2)}\circ J_\varphi),
		\end{equation}
		\begin{equation}\label{eq:fixed-background-drift-gauge}
			A_\varphi^{(1)ij}\partial_{ij}J_\varphi^\alpha+B_\varphi^{(1)i}\partial_iJ_\varphi^\alpha=\rho_\varphi(B_\varphi^{(2)\alpha}\circ J_\varphi),\quad\alpha=1,2,
		\end{equation}
		and
		\begin{equation}\label{eq:fixed-background-operator-intertwining}
			L_\varphi^{(1)}(f\circ J_\varphi)=\rho_\varphi(L_\varphi^{(2)}f)\circ J_\varphi\quad\text{for every }f\in C^\infty(\overline\Omega).
		\end{equation}
		Moreover,
		\begin{equation}\label{eq:boundary-gauge-data-section3}
			J_\varphi|_{\partial\Omega}=\Id,\quad \rho_\varphi|_{\partial\Omega}=1,\quad A_\varphi^{(1)}=A_\varphi^{(2)}\quad\text{on }\partial\Omega.
		\end{equation}
		For $\Sigma_j=K_jA_\varphi^{(j)}$, the same map satisfies \eqref{eq:conductivity-pushforward}.
	\end{proposition}
	
	\begin{proof}
		The calculation in \eqref{eq:first-linearized-divergence-form} uses only the curvature equation and the cofactor identity. It gives $L_\varphi^{(j)}=K_j^{-1}\operatorname{div}(\Sigma_j\nabla)$. Therefore the solutions defining $\Lambda'_j$ also define $\Lambda_{\Sigma_j}$, and $\Lambda_{\Sigma_j}h=K_j|_{\partial\Omega}\Lambda'_jh$. Lemma~\ref{lem:first-lin-Cauchy-data} and $K_1=K_2$ on $\partial\Omega$ give equality of these conductivity DN maps. More explicitly,
		\begin{equation}\label{eq:conductivity-data-from-nonlinear}
			\Lambda_{\Sigma_j}h=(\nu^T\Sigma_j\tau)\partial_\tau h+(\nu^T\Sigma_j\nu)D\Lambda_{K_j}(\varphi)[h].
		\end{equation}
		The boundary matrices agree by Lemma~\ref{lem:boundary-determination}. Equality on smooth data also gives equality as maps $H^{1/2}(\partial\Omega)\to H^{-1/2}(\partial\Omega)$, by their boundedness and density of smooth boundary functions. Lemma~\ref{lem:imported-planar-gauge} now gives the asserted conductivity transformation and all three operator identities. The boundary values in \eqref{eq:boundary-gauge-data-section3} follow from that lemma and the boundary reconstruction.
	\end{proof}
	
	\begin{corollary}[Boundary first jet of the gauge]\label{cor:boundary-first-jet-gauge}
		For the gauge in Proposition~\ref{thm:fixed-background-gauge}, $DJ_\varphi=I$ on $\partial\Omega$.
	\end{corollary}
	
	\begin{proof}
		Equations \eqref{eq:fixed-background-principal-gauge} and \eqref{eq:boundary-gauge-data-section3} give $NAN^T=A$ on the boundary, where $N=DJ_\varphi$ and $A=A_\varphi^{(1)}$. In the ordered orthonormal basis $(\tau,\nu)$, write $N=\left(\begin{smallmatrix}1&\alpha\\0&\beta\end{smallmatrix}\right)$ and $A=\left(\begin{smallmatrix}a&b\\b&c\end{smallmatrix}\right)$. Here $c>0$ and $\beta>0$, since the diffeomorphism maps the interior side to the interior side. Direct multiplication gives
		\[
		NAN^T=\begin{pmatrix}a+2\alpha b+\alpha^2c&\beta(b+\alpha c)\\\beta(b+\alpha c)&\beta^2c\end{pmatrix}.
		\]
		The lower right entry gives $\beta^2c=c$, so $\beta=1$. The off-diagonal entry then gives $b+\alpha c=b$, hence $\alpha=0$. This proves the claim at every boundary point.
	\end{proof}
	
	\begin{remark}
		The first-order coefficient in \eqref{eq:fixed-background-drift-gauge} does not transform by a pure vector-field pullback. The term $A_\varphi^{(1)ij}\partial_{ij}J_\varphi^\alpha$ comes from the chain rule for the principal part. The conductivity formulation accounts for this term through \eqref{eq:conductivity-divergence-transform}.
	\end{remark}
	
	\subsection{The conductivity form and the curvature ratio gauge equations}
	\label{subsec:first-linearized-conductivity-form}
	
	Fix the background boundary value and suppress the subscript $\varphi$. Recall that
	\begin{equation*}
		L_jv=A_j^{ab}\partial_{ab}v+B_j^a\partial_av,\quad A_j=(D^2u_j)^{-1},\quad B_j^a=-\frac{4\partial_au_j}{1+|\nabla u_j|^2}.
	\end{equation*}
	Let $\sigma_j=1+|\nabla u_j|^2$. Since $\Cof(D^2u_j)^{ab}=(\det D^2u_j)A_j^{ab}$ and $\det D^2u_j=K_j\sigma_j^2$, one has $K_jA_j^{ab}=\Cof(D^2u_j)^{ab}/\sigma_j^2$. The Piola identity gives $\partial_a\Cof(D^2u_j)^{ab}=0$, and hence,
	\begin{equation*}
		\partial_a(K_jA_j^{ab})=-2\sigma_j^{-3}(\partial_a\sigma_j)\Cof(D^2u_j)^{ab}.
	\end{equation*}
	Using $\partial_a\sigma_j=2\partial_m u_j\,\partial_{am}u_j$, we get
	\begin{equation*}
		\partial_a(K_jA_j^{ab})=-4\sigma_j^{-3}(\det D^2u_j)A_j^{ab}\partial_m u_j\,\partial_{am}u_j=-4K_j\frac{\partial_bu_j}{\sigma_j}=K_jB_j^b.
	\end{equation*}
	Thus,
	\begin{equation}\label{eq:first-linearized-divergence-form}
		L_jv=K_j^{-1}\partial_a(K_jA_j^{ab}\partial_bv).
	\end{equation}
	The associated conductivity is $\Sigma_j=K_jA_j$.
	
	The scalar density chosen in Lemma~\ref{lem:imported-planar-gauge} is also forced by the operator identity. We record the converse calculation for the same map $J$, using only the divergence form and the boundary normalization.
	
	\begin{lemma}[The scalar density from the divergence form]
		\label{lem:compatibility-first-linearized-gauges}
		Let $J$ and $\rho$ be the boundary fixing gauge obtained from \eqref{eq:imported-gauge-intertwining}. Then 
		\begin{equation}\label{eq:principal-gauge-rho-reduction}
			DJ\,A_1\,DJ^T=\rho\,A_2\circ J.
		\end{equation}
		Moreover,
		\begin{equation}\label{eq:rho-determinant-curvature-ratio}
			\rho
			=
			(\det DJ)\frac{K_2\circ J}{K_1}.
		\end{equation}
	\end{lemma}
	
	\begin{proof}
		Write
		\begin{equation}\label{eq:notation N d}
			N=DJ,
			\quad
			d=\det N,
			\quad
			\Sigma_j=K_jA_j.
		\end{equation}
		Since $J$ is a diffeomorphism and $DJ=I$ on $\partial\Omega$, we have
		$d>0$ on $\overline\Omega$.
		
		Comparing the second-order coefficients in the intertwining identity gives 	\eqref{eq:principal-gauge-rho-reduction}, namely $NA_1N^T=\rho(A_2\circ J)$. We now write the same identity in divergence form. Let $y=J(x)$. Set $\Sigma_1^J=(d^{-1}N\Sigma_1N^T)\circ J^{-1}$. The divergence transformation gives $\operatorname{div}_x(\Sigma_1\nabla_x(f\circ J))=d[\operatorname{div}_y(\Sigma_1^J\nabla_yf)]\circ J$. Thus,
		\begin{equation*}
			L_1(f\circ J)=K_1^{-1}d[\operatorname{div}_y(\Sigma_1^J\nabla_yf)]\circ J.
		\end{equation*}
		Set $\widetilde K_1=K_1\circ J^{-1}$, $\widetilde\rho=\rho\circ J^{-1}$ and $\widetilde d=d\circ J^{-1}$, using \eqref{eq:principal-gauge-rho-reduction}, the above identity becomes
		\begin{equation*}
			L_1(f\circ J)=\Big[\frac{\widetilde d}{\widetilde K_1}\operatorname{div}_y\Big(\frac{\widetilde K_1\widetilde\rho}{\widetilde d}A_2\nabla_y f\Big)\Big]\circ J.
		\end{equation*}
		On the other hand,
		\begin{equation*}
			\rho(L_2f)\circ J=\Big[ \frac{\widetilde\rho}{K_2}\operatorname{div}_y(K_2A_2\nabla_y f)\Big]\circ J.
		\end{equation*}
		Comparing the first-order coefficients gives
		\begin{equation*}
			A_2\nabla_y\log\Big(\frac{\widetilde K_1\widetilde\rho}{\widetilde d K_2}\Big)=0.
		\end{equation*}
		Since $A_2$ is invertible, the function $\frac{\widetilde K_1\widetilde\rho}{\widetilde d K_2}$	is constant. On $\partial\Omega$, we have $J=\Id$, $DJ=I$, $d=1$,
		$\rho=1$, and $K_1=K_2$. Hence, the constant is $1$. Returning to the
		$x$ variable gives $\frac{K_1\rho}{d(K_2\circ J)}=1$, which is exactly \eqref{eq:rho-determinant-curvature-ratio}.
	\end{proof}
	
	Adopting the notations \eqref{eq:notation N d}, define the affine-coordinate curvature ratio by
	\begin{equation}\label{eq:q-aff-d-theta-definitions}
		q_{\rm aff}:=\log\frac{K_1}{K_2\circ J},
		\quad
		\theta:=\log d-q_{\rm aff}.
	\end{equation}
	Then \eqref{eq:rho-determinant-curvature-ratio} gives $\rho=e^\theta$, $\rho^{-1}=e^{-\theta}$,
	and therefore $\mathbf r=1-\rho^{-1}=1-e^{-\theta}$.

	\subsection{The drift mismatch and the equation for \texorpdfstring{$J$}{J}}
	\label{subsec:drift-mismatch-q-equation}
	
	We now extract the equation for the gauge map. The affine-coordinate curvature ratio is $q_{\rm aff}=\log(K_1/(K_2\circ J))$. The background boundary value is fixed, and the subscript $\varphi$ is suppressed.
	
	Recall that $B_j^a=-4\partial_au_j/(1+|\nabla u_j|^2)$, and set $Z_j=-B_j$. Let $(g_j)_{ab}=\partial_{ab}u_j$ and $(g_j)^{ab}=A_j^{ab}$. Then $g_j$ is the Hessian metric of $u_j$, and $A_j=g_j^{-1}$. With the convention $\Delta_{g_j}v=A_j^{ab}\partial_{ab}v-A_j^{ab}\Gamma(g_j)^k_{ab}\partial_kv$, we have $L_j=\Delta_{g_j}+Y_j\cdot\nabla$ with $Y_j=X_{g_j}-Z_j$, where $X_{g_j}^k=A_j^{ab}\Gamma(g_j)^k_{ab}$. We first relate $Z_j$ to $X_{g_j}$ and $K_j$.

	Since $(g_j)_{ab}=\partial_{ab}u_j$, the Christoffel symbols of $g_j$ are
	\begin{equation*}
		\Gamma(g_j)^k_{ab}=\frac12 A_j^{k\ell}\big(\partial_a(g_j)_{b\ell}+\partial_b(g_j)_{a\ell}
		-\partial_\ell(g_j)_{ab}\big).
	\end{equation*}
	Substituting $(g_j)_{ab}=\partial_{ab}u_j$ gives $\Gamma(g_j)^k_{ab}=\frac12 A_j^{k\ell}\big(\partial_{ab\ell}u_j
	+\partial_{ba\ell}u_j-\partial_{\ell ab}u_j\big)$. By the symmetry of third derivatives, the first two terms are equal, and the last term is the same third derivative; then we obtain $\Gamma(g_j)^k_{ab}=\frac12 A_j^{k\ell}\partial_{ab\ell}u_j$. This implies 	\begin{equation*}
		X_{g_j}^k=A_j^{ab}\Gamma(g_j)^k_{ab}=\frac12 A_j^{k\ell}A_j^{ab}\partial_{ab\ell}u_j.
	\end{equation*}
	Since mixed derivatives commute, $\partial_{ab\ell}u_j=\partial_{\ell ab}u_j$, and we also have $\partial_\ell\log\det D^2u_j
	=A_j^{ab}\partial_{\ell ab}u_j$. Thus,
	\begin{equation*}
		X_{g_j}^k=\frac12 A_j^{k\ell}\partial_\ell\log\det D^2u_j.
	\end{equation*}
	Using $\det D^2u_j=K_j(1+|\nabla u_j|^2)^2$, we get $\partial_\ell\log\det D^2u_j=\partial_\ell\log K_j+2\partial_\ell\log(1+|\nabla u_j|^2)$.
	Moreover, since
	$\partial_\ell(1+|\nabla u_j|^2)=2\partial_m u_j\,\partial_{\ell m}u_j$, this becomes $\partial_\ell\log\det D^2u_j=\partial_\ell\log K_j
	+\frac{4\partial_m u_j\,\partial_{\ell m}u_j}{1+|\nabla u_j|^2}$. Substituting this into the formula for $X_{g_j}^k$ gives $X_{g_j}^k=\frac12A_j^{k\ell}\partial_\ell\log K_j+\frac{2A_j^{k\ell}\partial_{\ell m}u_j\,\partial_m u_j}{1+|\nabla u_j|^2}$. Since $A_j=(D^2u_j)^{-1}$, we have $A_j^{k\ell}\partial_{\ell m}u_j=\delta_m^k$. Hence,
	\begin{equation*}
		X_{g_j}^k=\frac12A_j^{k\ell}\partial_\ell\log K_j+\frac{2\partial_ku_j}{1+|\nabla u_j|^2}.
	\end{equation*}
	Since $Z_j^k=\frac{4\partial_ku_j}{1+|\nabla u_j|^2}$, we obtain
	\begin{equation}\label{eq:Z-X-K-identity}
		Z_j^k=2X_{g_j}^k-A_j^{k\ell}\partial_\ell\log K_j.
	\end{equation}
	
	Expanding the intertwining identity \eqref{eq:imported-gauge-intertwining} and comparing the coefficients of $\partial_\alpha f$ gives 
	\begin{equation}\label{eq:first-order-coefficient-J-Z}
		A_1^{ij}\partial_{ij}J^\alpha-Z_1^i\partial_iJ^\alpha=-\rho Z_2^\alpha(J).
	\end{equation}
	We also use the corresponding identity for the Laplace--Beltrami parts. The principal gauge relation is $g_1=\rho^{-1}J^*g_2$, where $\rho$ is evaluated in the source coordinates. By conformal covariance of the scalar Laplacian in dimension two,
	\begin{equation*}
		\Delta_{g_1}(f\circ J)=\rho(\Delta_{g_2}f)\circ J.
	\end{equation*}
	Comparing the coefficients of $\partial_\alpha f$ gives
	\begin{equation}\label{eq:LB-first-order-J}
		A_1^{ij}\partial_{ij}J^\alpha
		-X_{g_1}^i\partial_iJ^\alpha
		=-\rho X_{g_2}^\alpha(J).
	\end{equation}
	Subtracting \eqref{eq:LB-first-order-J} from
	\eqref{eq:first-order-coefficient-J-Z} yields
	\begin{equation}\label{eq:Z-X-mismatch}
		Z_1^i\partial_iJ^\alpha
		-\rho Z_2^\alpha(J)
		=X_{g_1}^i\partial_iJ^\alpha
		-\rho X_{g_2}^\alpha(J).
	\end{equation}
	
	Insert \eqref{eq:Z-X-K-identity} into both sides. The terms involving
	$2X_{g_j}$ cancel. Using the principal gauge relation, we obtain
	\begin{equation}\label{eq:Z-mismatch-q}
		Z_1^i\partial_iJ^\alpha-\rho Z_2^\alpha(J)=\partial_iJ^\alpha A_1^{ij}\partial_jq_{\rm aff}.
	\end{equation}
	Combining \eqref{eq:first-order-coefficient-J-Z} and \eqref{eq:Z-mismatch-q} gives, in the original Euclidean affine coordinate,
	\begin{equation}\label{eq:J-q-equation-affine}
		A_1^{ij}\partial_{ij}J^\alpha=\partial_iJ^\alpha A_1^{ij}\partial_jq_{\rm aff} .
	\end{equation}
	Equivalently, this is the tensorial identity
	\begin{equation}\label{eq:J-q-equation-tensorial}
		A_1^{ij}(\nabla dJ)^\alpha_{ij}=\partial_iJ^\alpha A_1^{ij}\partial_jq_{\rm aff},
	\end{equation}
	where the covariant Hessian of the map $J$ is taken with respect to the flat Euclidean connection in the domain and target.
	
	The identities above are in the original Euclidean coordinates. We keep these coordinates for the rigidity argument in Section~\ref{sec:ucp-and-gauge-elimination}. Local isothermal coordinates will be used only to prove unique continuation for a scalar principal operator; the components of the gradient difference will then be treated as scalar functions under that coordinate change.

	\section{Cauchy uniqueness for a local elliptic system}\label{sec:local-Cauchy-uniqueness}
	
	We prove the continuation result needed for the gradient difference. The essential condition is that all components have the same scalar elliptic principal part. The coupling may contain first derivatives, but not second derivatives of other components.
	
	\begin{lemma}[Unique continuation]\label{lem:local-elliptic-UCP}
		Let $O\subset\R^2$ be connected, and let $a=(a^{ij})$ be a smooth real symmetric matrix-valued function that is uniformly positive definite on compact subsets of $O$. Suppose that a real vector $v=(v_1,\ldots,v_m)\in H^1_{\loc}(O;\R^m)$ satisfies
		\begin{equation}\label{eq:common-principal-system}
			\partial_i(a^{ij}\partial_jv_\alpha)+b_{\alpha\beta}^i\partial_iv_\beta+d_{\alpha\beta}v_\beta=0,\quad \alpha=1,\ldots,m,
		\end{equation}
		with smooth real coefficients, where repeated spatial indices range over $1,2$ and component indices over $1,\ldots,m$. If $v$ vanishes on a nonempty open subset of $O$, then $v=0$ in $O$.
	\end{lemma}
	
	\begin{proof}
		First, $v$ is smooth in the interior. The right-hand side of each scalar divergence equation is initially in $L^2_{\loc}$, so local elliptic regularity gives $v_\alpha\in H^2_{\loc}$. Repeating this argument with the smooth coefficients gives $v\in H^k_{\loc}$ for every $k$, and hence $v\in C^\infty(O)$.
		
		In a sufficiently small coordinate neighborhood, choose isothermal coordinates for the metric $a^{-1}$; see \cite[Section 3.4]{PSU23}. If $y=\chi(x)$ and $w_\alpha=v_\alpha\circ\chi^{-1}$, then $D\chi\,a\,D\chi^T=\kappa I$ for a smooth positive $\kappa$. Apply the same change of variables to all components, treating each $v_\alpha$ as a scalar function. The second derivative chain rule gives $a^{ij}\partial_{ij}v_\alpha=\kappa(\Delta_yw_\alpha)\circ\chi+a^{ij}\partial_{ij}\chi^k(\partial_kw_\alpha)\circ\chi$. The remaining terms in \eqref{eq:common-principal-system} have order at most one. On a smaller relatively compact neighborhood, division by $\kappa$ gives
		\begin{equation}\label{eq:local-Laplacian-inequality}
			|\Delta w|\leq C(|Dw|+|w|).
		\end{equation}
		
		Write $z=y_1+\mathsf{i}y_2$, $\partial=(\partial_{y_1}-\mathsf{i}\partial_{y_2})/2$, and $\overline\partial=(\partial_{y_1}+\mathsf{i}\partial_{y_2})/2$. Set $Y=(w,\partial w)\in\C^{2m}$. Since each $w_\alpha$ is real, $\overline\partial w_\alpha=\overline{\partial w_\alpha}$ and $|Dw_\alpha|^2=4|\partial w_\alpha|^2$. Also, $\overline\partial\partial w_\alpha=\Delta w_\alpha/4$. Thus, \eqref{eq:local-Laplacian-inequality} implies
		\begin{equation}\label{eq:local-first-order-inequality}
			|\overline\partial Y|\leq M|Y|
		\end{equation}
		on each smaller coordinate disk, for a finite constant $M$.
		
		We prove the required continuation for this inequality. On a punctured disk, let $\Psi=\tau\log|z|^2-|z|^2/(2\epsilon)$, where $\epsilon>0$ and $\tau>0$. For a scalar test function $q$ compactly supported away from zero,
		\begin{equation}\label{eq:local-first-order-Carleman}
			\epsilon^{-1}\|e^{-\Psi}q\|_{L^2}^2\leq\|e^{-\Psi}\overline\partial q\|_{L^2}^2.
		\end{equation}
		Indeed, set $\mathcal P=e^{-\Psi}\overline\partial e^\Psi=\overline\partial+\overline\partial\Psi$. Its Hermitian adjoint is $\mathcal P^*=-\partial+\partial\Psi$. Since $\Delta\Psi=-2/\epsilon$ away from zero, direct commutation and integration by parts give
		\begin{equation}\label{eq:local-Carleman-commutator}
			\|\mathcal Pf\|_2^2-\|\mathcal P^*f\|_2^2=-\frac12\int\Delta\Psi\,|f|^2\,dA=\epsilon^{-1}\|f\|_2^2.
		\end{equation}
		Take $f=e^{-\Psi}q$ and discard the nonnegative term $\|\mathcal P^*f\|_2^2$. This proves \eqref{eq:local-first-order-Carleman}. The estimate also holds for compactly supported $H^1$ functions on the annulus by approximation, and for vectors by summation over the components.
		
		Suppose that $Y$ vanishes near the center of a coordinate disk $B_R$ on which \eqref{eq:local-first-order-inequality} holds. Translate its center to zero. Choose $0<R_0<R_1<R_2<R$ and $\eta\in C_c^\infty(B_{R_2})$ equal to one on $B_{R_1}$. The function $\eta Y$ vanishes near zero and is an admissible test function. Applying \eqref{eq:local-first-order-Carleman} componentwise gives
		\[
		\epsilon^{-1}\|e^{-\Psi}\eta Y\|_2^2\leq2M^2\|e^{-\Psi}\eta Y\|_2^2+2\|e^{-\Psi}(\overline\partial\eta)Y\|_2^2.
		\]
		Fix $\epsilon$ so that $2\epsilon M^2\leq1/2$. Absorption yields $\|e^{-\Psi}\eta Y\|_2^2\leq4\epsilon\|e^{-\Psi}(\overline\partial\eta)Y\|_2^2$. Since $e^{-2\Psi}=|z|^{-4\tau}e^{|z|^2/\epsilon}$, we obtain
		\begin{equation}\label{eq:local-zero-set-propagation}
			\|Y\|_{L^2(B_{R_0})}^2\leq4\epsilon\|\overline\partial\eta\|_\infty^2e^{R_2^2/\epsilon}\Big(\frac{R_0}{R_1}\Big)^{4\tau}\|Y\|_{L^2(B_{R_2}\setminus B_{R_1})}^2.
		\end{equation}
		All quantities except the displayed power are fixed. Letting $\tau\to\infty$ gives $Y=0$ on $B_{R_0}$. The radii were arbitrary, so the zero set fills $B_R$.
		
		Finally, let $Z$ be the set of points of $O$ having a neighborhood on which $v=0$. It is open and nonempty. If $x$ lies in its closure, choose one isothermal chart about $x$ whose image contains a closed disk about $\chi(x)$. A point of $Z$ sufficiently close to $x$ has image $y$ such that a disk centered at $y$, containing $\chi(x)$ in its interior, is still contained in this chart. The transformed vector $Y$ vanishes near $y$, and the preceding argument extends its vanishing to that disk. Thus, $x\in Z$. The set $Z$ is also closed in $O$, and connectedness proves $Z=O$.
	\end{proof}
	
	\begin{lemma}[Boundary unique continuation]\label{lem:local-elliptic-Cauchy}
		Let $\Omega\subset\R^2$ be bounded and connected with smooth boundary. Let $a$ be a smooth real symmetric uniformly positive definite matrix on $\overline\Omega$, and let $c=(c_1,c_2)$ be smooth and real. Suppose that $V\in C^\infty(\overline\Omega;\R^2)$ satisfies
		\begin{equation}\label{eq:local-gradient-system}
			\operatorname{div}(a\nabla V_\alpha)+\partial_\alpha(c\cdot V)=0,\quad\alpha=1,2,
		\end{equation}
		and that $V=0$ and $(a\nabla V_\alpha)\cdot\nu=0$ on a nonempty relatively open boundary arc, for both components. Then $V=0$ in $\Omega$.
	\end{lemma}
	
	\begin{proof}
		Fix a point in the stated boundary arc and a small ball $B$ about it that meets no other part of the boundary. The tangential derivative of each component is zero on the arc, so $\nabla V_\alpha=\nu\partial_\nu V_\alpha$ there. The conormal condition and $\nu^Ta\nu>0$ give $\partial_\nu V_\alpha=0$. Thus both $V$ and its first derivatives vanish on the arc.
		
		Extend $a,c$ smoothly to $B$, preserving positive definiteness of $a$ by taking $B$ smaller if necessary. Extend $V$ by zero from $B\cap\Omega$ to $B$, and denote the extension by $V^0$. The zero trace implies $V^0\in H^1(B)$; the vanishing first derivatives also show it is $C^1$ across the interface. For $\zeta\in C_c^\infty(B)$, integration by parts in $B\cap\Omega$ gives
		\begin{equation}\label{eq:zero-extension-weak-system}
			\begin{aligned}
				&\quad \, \int_B\big(a\nabla V_\alpha^0\cdot\nabla\zeta+(c\cdot V^0)\partial_\alpha\zeta\big)\,dx\\
				&=\int_{B\cap\partial\Omega}\zeta\big((a\nabla V_\alpha)\cdot\nu+(c\cdot V)\nu_\alpha\big)\,dS
				-\int_{B\cap\Omega}\zeta\big(\operatorname{div}(a\nabla V_\alpha)+\partial_\alpha(c\cdot V)\big)\,dx=0.
			\end{aligned}
		\end{equation}
		The interface terms vanish by the two Cauchy conditions. Thus, $V^0$ satisfies the same local system in distributions on the full ball; no boundary measure is produced. Expanding $\partial_\alpha(c\cdot V^0)$ gives only first-order and zeroth-order coupling, so the regularity argument and Lemma~\ref{lem:local-elliptic-UCP} apply. Since $V^0$ vanishes on the exterior open part of $B$, it vanishes throughout $B$. We have obtained an interior open zero set. Applying the same lemma to the connected domain $\Omega$ proves the assertion.
	\end{proof}

	\section{Rigidity of the conductivity gauge}\label{sec:ucp-and-gauge-elimination}
	
	The argument in this section uses the original Euclidean coordinates. For a vector-valued map $p$, our Jacobian convention is $(Dp)_{\alpha i}=\partial_i p_\alpha$. For a general invertible $2\times2$ matrix $N$, we use $\Cof N=(\det N)N^{-T}$. Hence, $\Cof$ denotes the cofactor matrix, not its transpose. In dimension two, $\Cof(\Cof N)=N$, $\Cof(tN)=t\Cof N$, and $\Cof(NM)=\Cof N\Cof M$. Define the positive gradient factor $F(p)=(1+|p|^2)^2$; below $F^{-1}$ means its reciprocal $1/F$.
	
	\begin{proposition}[Rigidity from the background gradients]\label{prop:residual-gauge-elimination}
		Let $\Omega\subset\R^2$ be bounded and connected with smooth boundary, and let $u_1,u_2\in C^\infty(\overline\Omega)$ have positive definite Hessians on $\overline\Omega$. Set
		\begin{equation}\label{eq:structured-conductivity}
			H_j=D^2u_j,\quad p_j=\nabla u_j,\quad \Sigma_j=\frac{\Cof H_j}{F(p_j)}.
		\end{equation}
		Suppose that $p_1=p_2$ on $\partial\Omega$ and that an orientation preserving smooth diffeomorphism $J:\overline\Omega\to\overline\Omega$ fixes the boundary and satisfies
		\begin{equation}\label{eq:rigidity-conductivity-gauge}
			\Sigma_2\circ J=d^{-1}N\Sigma_1N^T,\quad N=DJ,\quad d=\det N>0.
		\end{equation}
		Then $J=\Id$ and $\nabla u_1=\nabla u_2$ in $\Omega$. In particular, the two functions agree if they have the same Dirichlet trace.
	\end{proposition}
	
	\begin{proof}
		\textit{The pulled-back gradient.} Set
		\begin{equation}\label{eq:def_of_p_F}
		   p=p_2\circ J, \quad  V=p-p_1, \quad \text{and} \quad \beta=\frac{F(p_1)}{F(p)}.
		\end{equation} 
		These functions are smooth, and $\beta$ has a positive lower bound on $\overline\Omega$. The vector $p$ is the composition of the Euclidean gradient with $J$; it is not the gradient of $u_2\circ J$ in the source coordinates. Its derivative is $Dp=(H_2\circ J)N$.
		
		For a symmetric $2\times2$ matrix $H$, the identity $\Cof(\Cof H)=H$ gives $\Cof(\Cof H/F)=H/F$.  Therefore, taking cofactors in \eqref{eq:rigidity-conductivity-gauge} yields
		\begin{equation}\label{eq:cofactor-conductivity-gauge}
			\frac{H_2\circ J}{F(p)}=\frac1{dF(p_1)}(\Cof N)H_1(\Cof N)^T.
		\end{equation}
		Here the factor is $d^{-1}$, because the cofactor is homogeneous of degree one in dimension two. Multiply by $F(p)$ and then by $N$ on the right. Since $(\Cof N)^TN=dI$, we obtain
		\begin{equation}\label{eq:pulled-gradient-Jacobian}
			Dp=\beta^{-1}(\Cof N)H_1.
		\end{equation}
		Transposing and multiplying on the left give
		\begin{equation}\label{eq:pulled-gradient-flux}
			a(Dp)^T=(\Cof N)^T,\quad a=\beta H_1^{-1}.
		\end{equation}
		The matrix $a$ is smooth and uniformly positive definite on the closed domain.\\
		
		$\bullet$ \textit{The local differential system.}\\
		
		\noindent If $J=(J^1,J^2)$, the convention above gives
		\begin{equation}\label{eq:rowwise-Piola}
			\Cof DJ=\begin{pmatrix}\partial_2J^2&-\partial_1J^2\\-\partial_2J^1&\partial_1J^1\end{pmatrix},\quad
			\partial_i(\Cof DJ)_{\alpha i}=0,\quad\alpha=1,2.
		\end{equation}
		The two divergence identities are exactly the commutation of mixed second derivatives of $J^2$ and $J^1$. The column form of \eqref{eq:pulled-gradient-flux} is $a\nabla p_\alpha=(\Cof N)^Te_\alpha$. Taking divergence and using \eqref{eq:rowwise-Piola} gives $\operatorname{div}(a\nabla p_\alpha)=0$.
		
		Since $p_1=\nabla u_1$ and $H_1$ is symmetric, $\nabla(p_1)_\alpha=H_1e_\alpha$. Hence, $a\nabla(p_1)_\alpha=\beta e_\alpha$. Subtracting its divergence from the equation for $p_\alpha$ gives
		\begin{equation}\label{eq:gradient-difference-divergence}
			\operatorname{div}(a\nabla V_\alpha)=-\partial_\alpha\beta.
		\end{equation}
		The right-hand side depends linearly on $V$ and its first derivatives once the comparison pair is fixed. Indeed, the fundamental theorem of calculus along $p_1+tV$ gives
		\begin{equation}\label{eq:beta-divided-difference}
			\beta-1=c_\ell V_\ell,\quad
			c_\ell=F(p_1)\int_0^1\partial_{p_\ell}(F^{-1})(p_1+tV)\,dt.
		\end{equation}
		For the prescribed Gaussian curvature factor $F(p)=(1+|p|^2)^2$, we have 
		\begin{equation}\label{eq:Gaussian-divided-difference}
			c_\ell=-4(1+|p_1|^2)^2\int_0^1\frac{(p_1)_\ell+tV_\ell}{(1+|p_1+tV|^2)^3}\,dt.
		\end{equation}
		The denominator is positive everywhere, including where $V=0$. Thus $c$ is smooth on $\overline\Omega$, with no division by $|V|$ and no smallness requirement. Combining \eqref{eq:gradient-difference-divergence} and \eqref{eq:beta-divided-difference} gives
		\begin{equation}\label{eq:closed-gradient-system}
			\operatorname{div}(a\nabla V_\alpha)+\partial_\alpha(c\cdot V)=0,\quad\alpha=1,2.
		\end{equation}
		In non-divergence form, this reads
		\begin{equation}\label{eq:gradient-system-orders}
			a^{ij}\partial_{ij}V_\alpha=-(\partial_i a^{ij})\partial_jV_\alpha-c_\ell\partial_\alpha V_\ell-(\partial_\alpha c_\ell)V_\ell.
		\end{equation}
		Thus the two components have the same scalar second-order operator, with coupling only in lower derivatives. Although $a,c$ depend on the fixed solutions and $J$, they are fixed smooth coefficients when applying Lemma~\ref{lem:local-elliptic-Cauchy}. That lemma needs no bound independent of the comparison pair.\\
		
		$\bullet$ \textit{The boundary unique continuation.}\\
		
		\noindent Boundary fixing and $p_1=p_2$ on $\partial\Omega$ give $V=0$ and $\beta=1$ there. We verify the conormal data without assuming $DJ=I$ on the boundary. At a boundary point, $N\tau=\tau$. In the common ordered orthonormal tangent-normal frame, write
		\[
		N=\begin{pmatrix}1&b\\0&e\end{pmatrix},\quad \Cof N=\begin{pmatrix}e&0\\-b&1\end{pmatrix}.
		\]
		Consequently, $(\Cof N)\nu=\nu$. Taking the normal component of \eqref{eq:pulled-gradient-flux} gives $(a\nabla p_\alpha)\cdot\nu=e_\alpha\cdot(\Cof N)\nu=\nu_\alpha$. On the other hand, $a\nabla(p_1)_\alpha=\beta e_\alpha$ gives $(a\nabla(p_1)_\alpha)\cdot\nu=\nu_\alpha$ there. Subtraction yields
		\begin{equation}\label{eq:gradient-zero-Cauchy}
			V=0,\quad(a\nabla V_\alpha)\cdot\nu=0\quad\text{on }\partial\Omega,\quad\alpha=1,2.
		\end{equation}
		Applying Lemma~\ref{lem:local-elliptic-Cauchy} to \eqref{eq:closed-gradient-system} gives $V=0$ in $\Omega$.\\
		
		$\bullet$ \textit{Determination of the map.}\\
		
		\noindent Now, plugging $V=0$ into \eqref{eq:def_of_p_F} and using \eqref{eq:beta-divided-difference}, we have $p=p_1$, $\beta=1$, and $Dp=H_1$. Equation \eqref{eq:pulled-gradient-Jacobian} becomes $H_1=(\Cof N)H_1$. Invertibility of $H_1$ gives $\Cof N=I$. Taking cofactors again in dimension two gives $N=I$, so $D(J-\Id)=0$. Connectedness and the boundary value imply $J=\Id$. Hence, $p_2=p_1$, proving equality of the gradients. The function $u_2-u_1$ is constant, and a common Dirichlet trace makes that constant zero. No global inverse of either gradient map has been used.
	\end{proof}
	
	For the conductivities coming from the curvature equation, the scalar in this proof agrees with the density in Section~\ref{sec:first-linearization}. Taking determinants in \eqref{eq:rigidity-conductivity-gauge} gives $\det\Sigma_2\circ J=\det\Sigma_1$, while \eqref{eq:structured-conductivity} and the equation give $\det\Sigma_j=K_j/F(p_j)$. Therefore,
	\begin{equation}\label{eq:beta-curvature-density}
		\beta=\frac{K_1}{K_2\circ J}=\frac{\det DJ}{\rho}.
	\end{equation}
	This is an identity for the same $J$, not an additional condition on the gauge.
	
	\begin{proof}[Proof of Theorem~\ref{thm:main}]
		Fix $\varphi_0\in\mathcal U$ and write $u_j=u_j(\varphi_0)$. The background Dirichlet values are equal, and \eqref{DN map same} gives equality of their normal derivatives. Hence, $\nabla u_1=\nabla u_2$ on $\partial\Omega$. Together with $K_1=K_2$ there, Lemma~\ref{lem:boundary-determination} gives $D^2u_1=D^2u_2$ and consequently $\Sigma_1=\Sigma_2$ on the boundary.
		
		Proposition~\ref{prop:wellposedness} justifies differentiating \eqref{DN map same} at $\varphi_0$ in each smooth direction. Lemma~\ref{lem:first-lin-Cauchy-data} and \eqref{eq:conductivity-data-from-nonlinear} give $\Lambda_{\Sigma_1}=\Lambda_{\Sigma_2}$. The planar conductivity theorem, in the smooth form proved in Lemma~\ref{lem:imported-planar-gauge}, gives a smooth boundary fixing diffeomorphism satisfying \eqref{eq:conductivity-pushforward}.
		
		The curvature equation and \eqref{eq:first-linearized-divergence-form} give $\Sigma_j=\Cof(D^2u_j)/F(\nabla u_j)$. All hypotheses of Proposition~\ref{prop:residual-gauge-elimination} are now verified. It gives $J=\Id$ and $u_1=u_2$ in $\Omega$. Taking the prescribed curvature of this common graph yields
		\[
		K_1=\frac{\det D^2u_1}{(1+|\nabla u_1|^2)^2}=\frac{\det D^2u_2}{(1+|\nabla u_2|^2)^2}=K_2.
		\]
		This proves Theorem~\ref{thm:main} under its stated assumptions.
	\end{proof}
	
	\begin{remark}\label{rem:background-and-first-variation-data}
		The proof uses $K_1=K_2$ on $\partial\Omega$, equality $\Lambda_{K_1}(\varphi_0)=\Lambda_{K_2}(\varphi_0)$ at one common admissible boundary value, and equality $D\Lambda_{K_1}(\varphi_0)[h]=D\Lambda_{K_2}(\varphi_0)[h]$ for every smooth boundary direction $h$. It uses neither normal derivatives of the prescribed curvature nor a second derivative of the nonlinear DN map. In particular, it proves the conclusion under the original first boundary jet assumption. The first derivative is still a full boundary operator; it is not a single scalar measurement.
	\end{remark}

	\section*{Statements and Declarations}
	
	\para{Data availability statement}
	No datasets were generated or analyzed during the current study.
	
	\para{Conflict of interest}
	The author declares no conflict of interest.
	
	\para{Acknowledgments}
	The author is partially supported by the National Science and Technology Council (NSTC) of Taiwan under project 113-2628-M-A49-003. The author acknowledges financial support from the Alexander von Humboldt Foundation through the Henriette Herz Scouting Programme, hosted by Universit\"at Duisburg-Essen, Germany.

	\bibliographystyle{alpha}
	\bibliography{ref}
	
\end{document}